\documentclass[11pt,a4paper]{article}%
\usepackage{amsmath}
\usepackage{amsfonts}
\usepackage{amssymb}
\usepackage{graphicx}%
\providecommand{\U}[1]{\protect\rule{.1in}{.1in}}
\providecommand{\U}[1]{\protect\rule{.1in}{.1in}}
\newtheorem{theorem}{Theorem}

\newtheorem{corollary}[theorem]{Corollary}

\newenvironment{proof}[1][Proof]{\noindent\textbf{#1.} }{\ \rule{0.5em}{0.5em}}
\include {mak}
\begin{document}

\begin{center}
{\Large  The (}${ \it {r}_{1}{ ,\ldots, r}_{p}}$\uppercase{ )-Bell polynomials}

\ \ \

{\large Mohammed Said Maamra  and   Miloud Mihoubi}

USTHB, Faculty of Mathematics, Po. Box 32 El Alia 16111 Algiers, Algeria.

{\large mmaamra@usthb.dz  or  mmaamra@yahoo.fr}
\\
{\large mmihoubi@usthb.dz  or  miloudmihoubi@gmail.com}
\end{center}

$\ \ \ \newline$\noindent\textbf{Abstract.} In a previous paper, Mihoubi et al. introduced the
$\left(  r_{1},\ldots,r_{p}\right)  $-Stirling numbers and the $\left(
r_{1},\ldots,r_{p}\right)  $-Bell polynomials and gave some of their combinatorial and
algebraic properties. These numbers and polynomials generalize, respectively, the
$r$-Stirling numbers of the second kind introduced by Broder and the $r$-Bell
polynomials introduced by Mez\H{o}. In this paper, we prove that the $\left(
r_{1},\ldots,r_{p}\right)  $-Stirling numbers of the second kind are log-concave.
We also give generating functions and generalized recurrences related to the  $\left(  r_{1},\ldots,r_{p}\right)  $-Bell polynomials.

\noindent\textbf{Keywords. }The $(r_{1},\ldots r_{p})$-Bell polynomials; the
$(r_{1},\ldots r_{p})$-Stirling numbers; log-concavity; generalized recurrences; generating functions.

\noindent Mathematics Subject Classification 2010: 11B73; 05A10; 11B83.

\section{Introduction}
In 1984, Broder \cite{bro} introduced and studied the $r$-Stirling number of second kind ${ n \brace k}_{\!\!r},$ which counts the number of partitions of the set $\left[  n\right]  =\left\{
1,2,\ldots,n\right\}  $ into $k$ non-empty subsets such that the $r$ first elements are in distinct subsets.
In 2011, Mez\H{o} \cite{mez1} introduced and studied the $r$-Bell polynomials.
In 2012, Mihoubi et al. \cite{mih1} introduced and studied the $\left(  r_{1},\ldots,r_{p}\right)  $-Stirling number of second kind ${ n \brace k }_{\!\!r_{1},\ldots,r_{p}},$
which counts the number of partitions of the set $\left[
n\right]  $ into $k$ non-empty subsets such that the elements of each of the $p$ sets
$R_{1} :=\left\{  1,\ldots,r_{1}\right\}  ,$ \  $R_{2} :=\left\{  r_{1}+1,\ldots,r_{1}+r_{2}\right\}  ,$
\ldots, $R_{p}  :=\left\{  r_{1}+\cdots r_{p-1}+1,\ldots,r_{1}+\cdots+r_{p}\right\}$
are in distinct subsets.
\vskip 2pt
\noindent This work is motivated by the study of the $r$-Bell polynomials \cite{mez1}, the $(\ \!\!r_{1},\ldots,r_{p})$-Stirling numbers of the second kind \cite{mih1}, in which we may establish
\begin{itemize}
\item the log-concavity of the $\left(  r_{1},\ldots,r_{p}\right)  $-Stirling
numbers of the second kind,
\item a generalized recurrences for the $\left(  r_{1},\ldots,r_{p}\right)  $-Bell polynomials, and%
\item the ordinary generating functions of these numbers and polynomials.
\end{itemize}
To begin, by the symmetry of $\left(  r_{1},\ldots,r_{p}\right)  $-Stirling numbers respect
to $r_{1},\ldots,r_{p},$ let us to suppose $r_{1}\leq r_{2}\leq\cdots\leq r_{p}$ and throughout this paper,
we use the following notations and definitions%
\begin{align*}
\mathbf{r}_{p} &  :=\left(  r_{1},\ldots,r_{p}\right)  , \ \ \left\vert
\mathbf{r}_{p}\right\vert  :=r_{1}+\cdots+r_{p}, \\
P_{t}\left(  z;\mathbf{r}_{p}\right) & :=\left(  z+r_{p}\right)
^{t}\left(  z+r_{p}\right)  ^{\underline{r_{1}}}\cdots\left(  z+r_{p}\right)
^{\underline{r_{p-1}}},\ \ t\in\mathbb{R},\\
B_{n}\left(  z;\mathbf{r}_{p}\right)   &  :=\underset{k=0}{\overset
{n+\left\vert \mathbf{r}_{p-1}\right\vert }{\sum}}%
{ n+\left\vert \mathbf{r}_{p}\right\vert \brace k+r_{p} }
_{\!\!\mathbf{r}_{p}}z^{k},\ \ n\geq0
\end{align*}
and $\mathbf{e}_{i}$ denote the i-$th$ vector of the canonical basis of $\mathbb{R}^{p}.$ \\
In \cite{mih1}, we have proved the following
\begin{align}%
B_{n}\left(  z;\mathbf{r}_{p}\right)  &=\exp\left(  -z\right)  \underset
{k\geq0}{\sum}P_{n}\left(  k;\mathbf{r}_{p}\right)  \frac{z^{k}}{k!}\label{9}\\
P_{n}\left(  z;\mathbf{r}_{p}\right) & =\underset{k=0}{\overset{n+\left\vert
\mathbf{r}_{p-1}\right\vert }{\sum}}%
{ n+\left\vert \mathbf{r}_{p}\right\vert \brace k+r_{p} }
_{\!\!\mathbf{r}_{p}}z^{\underline{k}}.\label{7}%
\end{align}
The following introduced numbers $a_{k}\left(  \mathbf{r}_{p-1}\right)$ will be used later:
\[
a_{k}\left(  \mathbf{r}_{p-1}\right)  =\left(  -1\right)  ^{\left\vert
\mathbf{r}_{p-1}\right\vert -k}\underset{\left\vert \mathbf{j}_{p-1}%
\right\vert =k}{\sum}\genfrac{[}{]}{0pt}{}{r_{1}}{j_{1}}\cdots%
\genfrac{[}{]}{0pt}{}{r_{p-1}}{j_{p-1}},\ \ \ \left\vert \mathbf{j}_{p-1}\right\vert =j_{1}+\cdots+j_{p-1}.
\]
where $\genfrac{[}{]}{0pt}{}{n}{k}$ are the absolute Stirling numbers of the first kind. The exponential generating function of the sequence
$(a_{k}\left(  \mathbf{r}_{p-1}\right); k\geq0)$ shows that we have
\begin{equation}
\underset{k=0}{\overset{\left\vert \mathbf{r}_{p-1}\right\vert }{\sum}}%
a_{k}\left(  \mathbf{r}_{p-1}\right)  u^{k}=\left(  u\right)  ^{\underline
{r_{1}}}\cdots\left(  u\right)  ^{\underline{r_{p-1}}}.\label{10}%
\end{equation}
On our contribution, we give more properties for the $\mathbf{r}_{p}$-Stirling
numbers and $\mathbf{r}_{p}$-Bell Polynomials organized as follows:
\\In the next section we generalize some results of those given by Mez\H{o} \cite{mez2}, in which, we
prove that the sequence $\left ({n+\left\vert \mathbf{r}_{p}\right\vert \brace k+r_{p}}_{\!\!%
\mathbf{r}_{p}};\ 0\leq k\leq n+\left\vert \mathbf{r}_{p-1}\right\vert \right )$
is strongly log-concave and we give an approximation of
${n+\left\vert \mathbf{r}_{p}\right\vert \brace k+r_{p}}_{\!\!\mathbf{r}_{p}}$ when $%
n\rightarrow \infty $ for fixed $k$. In the third section we generalize some results of those given by Mihoubi et al. \cite{mih2}, in which, we write $B_{n}\left(
z;\mathbf{r}_{p}\right) $ in the basis $\left\{ B_{n+k}\left( z;r_{p}\right) :0\leq k\leq \left\vert \mathbf{r}_{p-1}\right\vert \right\}$
and $B_{n+m}\left( z;\mathbf{r}_{p}\right) $ in the family of bases
$\left\{ z^{j}B_{m}\left( z;\mathbf{r}_{p}+j\mathbf{e}_{p}\right) :0\leq j\leq n\right\} .$
As consequences, we also give some identities for the $\mathbf{r}_{p}$-Stirling numbers.
In the fourth section we give the ordinary generating functions of the $\mathbf{r}_{p}$-Stirling numbers of the second kind and the $\mathbf{r}_{p}$-Bell polynomials.
\section{Log-concavity of the $\mathbf{r}_{p}$-Stirling numbers}
In this section, we discuss the real roots of the polynomial $B_{n}\left(  z;\mathbf{r}_{p}\right),$ the log-concavity of the sequence $\left({ n+\left\vert \mathbf{r}_{p}\right\vert \brace k+r_{p} }
_{\!\!\mathbf{r}_{p}},\ 0\leq k\leq n+\left\vert \mathbf{r}_{p-1}\right\vert
\right)  ,$  the greatest maximizing index of ${ n \brace k }
_{\!\!\mathbf{r}_{p}}$ and we give an approximation of ${ n+\left\vert \mathbf{r}_{p}\right\vert \brace m+r_{p} }
_{\!\!\mathbf{r}_{p}}$ when $n$ tends to infinity. The case $p=1$ was studied by Mez\H{o} \cite{mez2} and other study is given by Zhao \cite{zha} on a large class of Stirling numbers.
\begin{theorem}
\label{T1}The roots of the polynomial $B_{n}\left(  z;\mathbf{r}_{p}\right)$ are real and negative.
\end{theorem}
\begin{proof}
We will show by induction on $p$ that the roots of the polynomials $B_{n}\left( z;%
\mathbf{r}_{p}\right) $ are real and negative. Firstly, we may use the following polynomials
\begin{equation*}
B_{n}^{\left( j\right) }\left( z;\mathbf{r}_{p}\right) :=\exp \left(
-z\right) \frac{d^{j}}{dz^{j}}\left( z^{r_{p}}\exp \left( z\right)
B_{n}\left( z;\mathbf{r}_{p}\right) \right) ,\ \ 0\leq j\leq r_{p+1}.
\end{equation*}%
We note that these polynomials are of the same degree $n+\left\vert \mathbf{r}_{p}\right\vert $ which
satisfy%
\begin{equation*}
B_{n}^{\left( j+1\right) }\left( z;\mathbf{r}_{p}\right) =\exp \left(
-z\right) \frac{d}{dz}\left( \exp \left( z\right) B_{n}^{\left( j\right)
}\left( z;\mathbf{r}_{p}\right) \right)
\end{equation*}%
and, for  $0\leq j\leq r_{p},$ there exists a polynomial $Q_{n,j}\left( z;\mathbf{r}_{p}\right) $ such
that
\[
B_{n}^{\left( j\right) }\left( z;\mathbf{r}_{p}\right)=z^{r_{p}-j}Q_{n,j}\left( z;\mathbf{r}_{p}\right).
\]
Furthermore, we may establish the following
\begin{equation}
B_{n}\left( z;\mathbf{r}_{p+1}\right) =\exp \left( -z\right) \frac{%
d^{r_{p+1}}}{dz^{r_{p+1}}}\left( z^{r_{p}}\exp \left( z\right) B_{n}\left( z;%
\mathbf{r}_{p}\right) \right) .  \label{2}
\end{equation}%
Indeed, by (\ref{7}) we get
\begin{align*}
\frac{d^{r_{p+1}}}{dz^{r_{p+1}}}\left( z^{r_{p}}\exp \left( z\right)B_{n}\left( z;\mathbf{r}_{p}\right) \right)
=\frac{d^{r_{p+1}}}{dz^{r_{p+1}}}\left( \underset{k\geq 0}{\sum }P_{n}\left( k;\mathbf{r}_{p}\right)\frac{z^{k+r_{p}}}{k!}\right)
\end{align*}%
and this can be written as
\begin{align*}
\underset{k\geq r_{p+1}-r_{p}}{\sum }P_{n}\left( k;\mathbf{r}_{p}\right)\left(  k+r_{p}\right)  ^{\underline{r_{p+1}}}\frac{z^{k+r_{p}-r_{p+1}}}{k!}
& =\underset{k\geq 0}{\sum }P_{n}\left( k;\mathbf{r}_{p+1}\right)\frac{z^{k}}{k!}
\\ &=\exp \left( z\right) B_{n}\left( z;\mathbf{r}_{p+1}\right) .
\end{align*}%
Secondly, the proof by induction on $p$ is as follows: \\ For $p=0$ we set  $B_{n}\left( z;\mathbf{r}%
_{p}\right):= B_{n}\left( z\right)$ which is the classical Bell polynomial and for $p=1$ the
polynomial $B_{n}\left( z;r_{1}\right) $ is the $r_{1}$-Bell polynomial introduced
in \cite{mez1}. It is known that these polynomials have only real (negative) roots.
Assume, for $1\leq j\leq p,$ that the roots of the polynomial $B_{n}\left(  z;\mathbf{r}_{j}\right)$ are real and negative.
By these statements, the polynomial $z^{r_{p}}B_{n}\left( z;\mathbf{r}%
_{p}\right) $ has $n+\left\vert \mathbf{r}_{p-1}\right\vert +1$
non-positive real roots. So Rolle's theorem shows that the polynomial%
\begin{equation*}
B_{n}^{\left( 1\right) }\left( z;\mathbf{r}_{p}\right) =\exp \left(
-z\right) \frac{d}{dz}\left( z^{r_{p}}\exp \left( z\right) B_{n}\left( z;%
\mathbf{r}_{p}\right) \right) =z^{r_{p}-1}Q_{n,1}\left( z;\mathbf{r}%
_{p}\right)
\end{equation*}%
has $n+\left\vert \mathbf{r}_{p-1}\right\vert $ real (negative) roots and because $Q_{n,1}\left( z;%
\mathbf{r}_{p}\right) $ is of degree $n+\left\vert \mathbf{r}%
_{p-1}\right\vert +1,$ the last root may be necessarily real and negative.
This means that the polynomial $B_{n}^{\left( 1\right) }\left( z;\mathbf{r}%
_{p}\right) $ has $n+\left\vert \mathbf{r}_{p-1}\right\vert +1$
real (negative) roots and $z=0$ is a root of multiplicity $r_{p}-1.$
Similarly, Rolle's theorem shows that the polynomial%
\begin{equation*}
B_{n}^{\left( 2\right) }\left( z;\mathbf{r}_{p}\right) =\exp \left(
-z\right) \frac{d}{dz}\left( \exp \left( z\right) B_{n}^{\left( 1\right)
}\left( z;\mathbf{r}_{p}\right) \right) =z^{r_{p}-2}Q_{n,2}\left( z;\mathbf{r%
}_{p}\right)
\end{equation*}%
has $n+\left\vert \mathbf{r}_{p-1}\right\vert +1$ real (negative)
roots and because $%
Q_{n,2}\left( z;\mathbf{r}_{p}\right) $ is of degree $n+\left\vert
\mathbf{r}_{p-1}\right\vert +2,$ the last root may be necessarily real and
negative. This means that the polynomial $B_{n}^{\left( 2\right) }\left( z;%
\mathbf{r}_{p}\right) $ has $n+\left\vert \mathbf{r}_{p-1}\right\vert +2$
real (negative) roots and $z=0$ is a root of multiplicity $r_{p}-2.$%
We proceed similarly to conclude that, for $j\in \left\{ 0,\ldots
,r_{p}-1\right\} ,$ the polynomial $B_{n}^{\left( j\right) }\left( z;\mathbf{r%
}_{p}\right) $ has $n+\left\vert \mathbf{r}_{p-1}\right\vert +j$
real (negative) roots and $z=0$ is a root of multiplicity $r_{p}-j.$
For $r_{p}\leq j\leq r_{p+1},$ the polynomials $B_{n}^{\left( j\right) }\left( z;\mathbf{r}_{p}\right)
$ no have $z=0$ as a root. By applying Rolle's theorem on $B_{n}^{\left( r_{p}-1\right) }\left( z;\mathbf{r}%
_{p}\right) $ (which has $n+\left\vert \mathbf{r}_{p}\right\vert$ real non-positive roots) we conclude that $B_{n}^{\left( r_{p}\right) }\left( z;%
\mathbf{r}_{p}\right) $ has $n+\left\vert \mathbf{r}_{p}\right\vert -1$ real (negative) roots and because it is of degree $n+\left\vert \mathbf{r}%
_{p}\right\vert ,$ the last root may be necessarily real and negative. Similarly, we conclude that the
polynomial $B_{n}^{\left( r_{p}+1\right) }\left( z;\mathbf{r}_{p}\right) $
has $n+\left\vert \mathbf{r}_{p}\right\vert $ real (negative) roots and so on.
\end{proof}

\noindent On using Newton's inequality \cite[p. 52]{har} given by
\begin{theorem} (Newton's inequality) Let $a_{0},$ $a_{1},\ldots ,$ $a_{n}$ be real numbers.
If all the zeros of the polynomial $P(x)=\overset{n}{\underset{k=0}{\sum }}%
a_{i}x^{i}$ are real, then the coefficients of $P$ satisfy%
\begin{equation*}
a_{i}^{2}\geq \left( 1+\frac{1}{i}\right) \left( 1+\frac{1}{n-i}\right)
a_{i+1}a_{i-1},\ \ \ 1\leq i\leq n-1,
\end{equation*}
\end{theorem}
\noindent we may state that:
\begin{corollary}
The sequence $\left\{{ n+\left\vert \mathbf{r}_{p}\right\vert \brace k+r_{p} }
_{\!\!\mathbf{r}_{p}},\ 0\leq k\leq n+\left\vert \mathbf{r}_{p-1}\right\vert
\right\}  $ is strongly log-concave (and thus unimodal).
\end{corollary}
\noindent This property shows that the sequence $({ n \brace k }
_{\!\!\mathbf{r}_{p}}, 0\leq k \leq n)$  admits an index $K\in \{0,1,\ldots,n\}$ for which
${ n \brace K }_{\!\!\mathbf{r}_{p}}$ being the maximum of ${ n \brace k }_{\!\!\mathbf{r}_{p}}$.
An application of Darroch's inequality \cite{dar} will help us to localize this index.
\begin{theorem} (Darroch's inequality) Let $a_{0},$ $a_{1},\ldots ,$ $a_{n}$ be real numbers.
If all the zeros of the polynomial $P(x)=\overset{n}{\underset{k=0}{\sum }}%
a_{i}x^{i}$ are real and negative and $P(1) > 0,$ then the value of $k$ for which $a_{k}$ is maximized is within
one of $P'(1)/P(1)$.
\end{theorem}

\noindent The following corollary gives a small interval for this index.

\begin{corollary}
Let $K_{n,\mathbf{r}_{p}}$ be the greatest maximizing index of ${ n \brace k }
_{\!\!\mathbf{r}_{p}}.$  We have%
\[
\left\vert K_{n+\left\vert \mathbf{r}_{p}\right\vert ,\mathbf{r}_{p}}-\left(
\frac{B_{n+1}\left(  1;\mathbf{r}_{p}\right)  }{B_{n}\left(  1;\mathbf{r}%
_{p}\right)  }-\left(  r_{p}+1\right)  \right)  \right\vert <1.
\]

\end{corollary}

\begin{proof}
Since the sequence ${ n+\left\vert \mathbf{r}_{p}\right\vert \brace k+r_{p} }
_{\!\!\mathbf{r}_{p}}$ is strongly log-concave, there exists an index
$K_{n+\left\vert \mathbf{r}_{p}\right\vert ,\mathbf{r}_{p}}$ for which \
${ n+\left\vert \mathbf{r}_{p}\right\vert \brace r_{p} }
_{\!\!\mathbf{r}_{p}}<\cdots<{ n+\left\vert \mathbf{r}_{p}
\right\vert \brace K_{n+\left\vert \mathbf{r}_{p}\right\vert ,\mathbf{r}_{p}} }
_{\!\!\mathbf{r}_{p}}>\cdots>{ n+\left\vert \mathbf{r}_{p}\right\vert \brace n+r_{p} }
_{\!\!\mathbf{r}_{p}}.$ Then, on applying Theorem \ref{T1} and Darroch's theorem, we obtain%
\[
\left\vert K_{n+\left\vert \mathbf{r}_{p}\right\vert ,\mathbf{r}_{p}}%
-\frac{\left.  \frac{d}{dz}B_{n}\left(  z;\mathbf{r}_{p}\right)  \right\vert
_{z=1}}{B_{n}\left(  1;\mathbf{r}_{p}\right)  }\right\vert <1.
\]
It remains to apply the first identity given in \cite[Corollary 12]{mih1} by%
\begin{align*}
z\frac{d}{dz}\left(  z^{r_{p}}\exp\left(  z\right)  B_{n}\left(
z;\mathbf{r}_{p}\right)  \right)   &  =z^{r_{p}}\exp\left(  z\right)
B_{n+1}\left(  z;\mathbf{r}_{p}\right)  ,
\end{align*}
which is equivalent to $z\frac{d}{dz}\left(  B_{n}\left(  z;\mathbf{r}_{p}\right)  \right)
=B_{n+1}\left(  z;\mathbf{r}_{p}\right)  -\left(  z+r_{p}\right)  B_{n}\left(
z;\mathbf{r}_{p}\right)  .$
\end{proof}

\section{Generalized recurrences and consequences}
In this section, different representations of the polynomial $B_{n}\left(  z;\mathbf{r}_{p}\right)  $  in different bases or families of basis are given by Theorems \ref{T3} and \ref{T6}. Indeed, this polynomial admits a representation in the basis $\left\{  B_{n+k}\left(  z;r_{p}\right)  :\ 0\leq k\leq n+\left\vert \mathbf{r}_{p-1}\right\vert \right\}$ given by the following theorem.
\begin{theorem}
\label{T3} We have%
\begin{align*}
B_{n}\left(  z;\mathbf{r}_{p}\right) & =\underset{k=0}{\overset{\left\vert
\mathbf{r}_{p-1}\right\vert }{\sum}}a_{k}\left(  \mathbf{r}_{p-1}\right)
B_{n+k}\left(  z;r_{p}\right) , \\
B_{n}\left(  z;\mathbf{r}_{p+q}\right) & =\underset{k=0}{\overset{\left\vert
\mathbf{r}_{p-1}\right\vert }{\sum}}a_{k}\left(  \mathbf{r}_{p-1}\right)
B_{n+k}\left(  z;r_{p},\ldots,r_{p+q}\right)  .
\end{align*}

\end{theorem}

\begin{proof}
Upon using the fact that
$\left(  k+r_{p}\right)  ^{\underline{r_{m}}}=\underset{j=0}{\overset{r_{m}%
}{\sum}}\left(  -1\right)  ^{r_{m}-j} \genfrac{[}{]}{0pt}{}{r_{m}}{j}\left(  k+r_{p}\right)  ^{j}, \label{cc}$
we get%
\begin{align*}
B_{n}\left(  z;\mathbf{r}_{p}\right)   &  =\exp\left(  -z\right)
\underset{k\geq0}{\sum}P_{n}\left(  k;\mathbf{r}_{p}\right)  \frac{z^{k}}%
{k!}\\
&  =\exp\left(  -z\right)  \underset{k\geq0}{\sum}\underset{j=0}%
{\overset{r_{m}}{\sum}}\left(  -1\right)  ^{r_{m}-j}\genfrac{[}{]}{0pt}{}{r_{m}}{j}%
\frac{P_{0}\left(  k;\mathbf{r}_{p}\right)  }{\left(  k+r_{p}\right)
^{\underline{r_{m}}}}\left(  k+r_{p}\right)  ^{n+j}\frac{z^{k}}{k!}\\
&  =\underset{j=0}{\overset{r_{m}}{\sum}}\left(  -1\right)  ^{r_{m}-j}%
\genfrac{[}{]}{0pt}{}{r_{m}}{j}%
B_{n+j}\left(  z;\mathbf{r}_{p}-r_{m}\mathbf{e}_{m}\right)
,\ \ \ m=1,2,\ldots,p-1,
\end{align*}
and with the same process, we obtain
\begin{align*}
B_{n}\left(  z;\mathbf{r}_{p}\right) & =\underset{j_{1}=0}{\overset{r_{1}}{\sum}}\cdots\underset{j_{p-1}%
=0}{\overset{r_{p-1}}{\sum}}\left(  -1\right)  ^{\left\vert \mathbf{r}%
_{p-1}\right\vert -\left\vert \mathbf{j}_{p-1}\right\vert }%
\genfrac{[}{]}{0pt}{}{r_{1}}{j_{1}}\cdots%
\genfrac{[}{]}{0pt}{}{r_{p-1}}{j_{p-1}}%
B_{n+\left\vert \mathbf{j}_{p-1}\right\vert }\left(  z;r_{p}\right)
\\
&=\underset{k=0}{\overset{\left\vert \mathbf{r}_{p-1}\right\vert }{\sum}%
}\left(  -1\right)  ^{\left\vert \mathbf{r}_{p-1}\right\vert -k}B_{n+k}\left(
z;r_{p}\right)  \underset{\left\vert \mathbf{j}_{p-1}\right\vert =k}{\sum}%
\genfrac{[}{]}{0pt}{}{r_{1}}{j_{1}}\cdots%
\genfrac{[}{]}{0pt}{}{r_{p-1}}{j_{p-1}}%
\\
&  =\underset{k=0}{\overset{\left\vert \mathbf{r}_{p-1}\right\vert }{\sum}%
}a_{k}\left(  \mathbf{r}_{p-1}\right)  B_{n+k}\left(  z;r_{p}\right)  .
\end{align*}
This implies the first identity of the theorem.
\\ Now, from identity (\ref{2})  we can write
\begin{align*}
\exp \left( -z\right) \frac{d^{r_{p+1}}}{dz^{r_{p+1}}}\left( z^{r_{p}}\exp
\left( z\right) B_{n}\left( z;\mathbf{r}_{p}\right) \right) \\ =\underset{k=0}{%
\overset{\left\vert \mathbf{r}_{p-1}\right\vert }{\sum }}a_{k}\left( \mathbf{%
r}_{p-1}\right) \exp \left( -z\right) \frac{d^{r_{p+1}}}{dz^{r_{p+1}}}\left(
z^{r_{p}}\exp \left( z\right) B_{n+k}\left( z;r_{p}\right) \right) ,
\end{align*}%
which gives on utilizing (\ref{2}):
$B_{n}\left( z;\mathbf{r}_{p+1}\right) =\underset{k=0}{\overset{\left\vert
\mathbf{r}_{p-1}\right\vert }{\sum }}a_{k}\left( \mathbf{r}_{p-1}\right)
B_{n+k}\left( z;r_{p},r_{p+1}\right).$
\\ We can repeat this process $q$ times to obtain the second identity of the theorem.
\end{proof}

\noindent So, the $\left(  r_{1},\ldots,r_{p}\right)  $-Stirling numbers admit an expression in terms of the usual $r$-Stirling numbers given by the following corollary.
\begin{corollary}
\label{C1}We have%
\[%
{ n+\left\vert \mathbf{r}_{p}\right\vert \brace k+r_{p} }
_{\!\!\mathbf{r}_{p}}=\underset{j=0}{\overset{\left\vert \mathbf{r}_{p-1}%
\right\vert }{\sum}}%
{ n+j+r_{p} \brace k+r_{p} }
_{\!\!{r}_{p}}a_{j}\left(  \mathbf{r}_{p-1}\right)  .
\]

\end{corollary}

\begin{proof}
Using Theorem \ref{T3}, the polynomial $B_{n}\left(  z;\mathbf{r}_{p}\right)$ can be written as follows:
\begin{align*}
\underset{j=0}{\overset{\left\vert \mathbf{r}_{p-1}\right\vert }{\sum}}a_{j}\left(  \mathbf{r}_{p-1}\right)
B_{n+j}\left(  z;r_{p}\right) &  =\underset{j=0}{\overset{\left\vert \mathbf{r}_{p-1}\right\vert }{\sum}%
}a_{j}\left(  \mathbf{r}_{p-1}\right)  \underset{k=0}{\overset{n+j}{\sum}}%
{ n+j+r_{p} \brace k+r_{p} }_{\!\!{r}_{p}}z^{k}\\
&  =\underset{k=0}{\overset{n+\left\vert \mathbf{r}_{p-1}\right\vert }{\sum}%
}z^{k}\underset{j=0}{\overset{\left\vert \mathbf{r}_{p-1}\right\vert }{\sum}%
}a_{j}\left(  \mathbf{r}_{p-1}\right)
{ n+j+r_{p} \brace k+r_{p} }_{\!\!{r}_{p}}%
\end{align*}
and because $B_{n}\left(  z;\mathbf{r}_{p}\right)  =\underset{k=0}%
{\overset{n+\left\vert \mathbf{r}_{p-1}\right\vert }{\sum}}%
{ n+\left\vert \mathbf{r}_{p}\right\vert \brace k+r_{p} }
_{\!\!\mathbf{r}_{p}}z^{k},$ the identity follows by identification.
\end{proof}

\noindent In \cite{mih1}, we have proved the following
\[
\underset{n\geq0}{\sum}B_{n}\left(  z;\mathbf{r}_{p}\right)  \frac{t^{n}%
}{n!}=B_{0}\left(  z\exp\left(  t\right)  ;\mathbf{r}_{p}\right)  \exp\left(
z\left(  \exp\left(  t\right)  -1\right)  +r_{p}t\right) .
\]
The following theorem gives more details on the exponential generating function of the $\mathbf{r}_{p}$-Bell polynomials and will be used later.
\begin{theorem}
\label{T4}We have%
\begin{align*}
\underset{n\geq0}{\sum}B_{n+m}\left(  z;\mathbf{r}_{p}\right)  \frac{t^{n}%
}{n!}&=B_{m}\left(  z\exp\left(  t\right)  ;\mathbf{r}_{p}\right)  \exp\left(
z\left(  \exp\left(  t\right)  -1\right)  +r_{p}t\right) \\
&=\underset{k=0}{\overset{\left\vert
\mathbf{r}_{p-1}\right\vert }{\sum}}a_{k}\left(  \mathbf{r}_{p-1}\right)
\frac{d^{m+k}}{dt^{m+k}}\left(  \exp\left(  z\left(  \exp\left(  t\right)
-1\right)  +r_{p}t\right)  \right)  .
\end{align*}
\end{theorem}
\begin{proof}
Use (\ref{9}) to get%
\begin{align*}
\underset{n\geq0}{\sum}B_{n+m}\left(  z;\mathbf{r}_{p}\right)  \frac{t^{n}%
}{n!} &  =\underset{n\geq0}{\sum}\left(  \exp\left(  -z\right)  \underset{k\geq
0}{\sum}P_{0}\left(  k;\mathbf{r}_{p}\right)  \left(  k+r_{p}\right)
^{n+m}\frac{z^{k}}{k!}\right)  \frac{t^{n}}{n!}\\
&  =\exp\left(  -z\right)  \underset{k\geq0}{\sum}P_{0}\left(  k;\mathbf{r}%
_{p}\right)  \left(  k+r_{p}\right)  ^{m}\frac{z^{k}\exp\left(  \left(
k+r_{p}\right)  t\right)  }{k!}\\
&  =B_{m}\left(  z\exp\left(  t\right)  ;\mathbf{r}_{p}\right)  \exp\left(
z\left(  \exp\left(  t\right)  -1\right)  +r_{p}t\right)  .
\end{align*}
For the second part of the theorem, use Theorem \ref{T3} to obtain%
\begin{align*}
\underset{n\geq0}{\sum}B_{n+m}\left(  z;\mathbf{r}_{p}\right)  \frac{t^{n}}{n!}
&  =\underset{k=0}{\overset{\left\vert \mathbf{r}_{p-1}\right\vert }{\sum}%
}a_{k}\left(  \mathbf{r}_{p-1}\right)  \underset{n\geq0}{\sum}B_{n+m+k}\left(
z;r_{p}\right)  \frac{t^{n}}{n!}\\
&  =\underset{k=0}{\overset{\left\vert \mathbf{r}_{p-1}\right\vert }{\sum}%
}a_{k}\left(  \mathbf{r}_{p-1}\right)  \frac{d^{m+k}}{dt^{m+k}}\left(
\underset{n\geq0}{\sum}B_{n}\left(  z;r_{p}\right)  \frac{t^{n}}{n!}\right) \\
&  =\underset{k=0}{\overset{\left\vert \mathbf{r}_{p-1}\right\vert }{\sum}%
}a_{k}\left(  \mathbf{r}_{p-1}\right)  \frac{d^{m+k}}{dt^{m+k}}\left(
\exp\left(  z\left(  \exp\left(  t\right)  -1\right)  +r_{p}t\right)  \right).
\end{align*}
\end{proof}

\noindent Spivey \cite{spy} gave a beautiful combinatorial identity; after
that, in different ways, Belbachir et al. \cite{bel3},  Gould et al. \cite{gou},
generalized this identity on showing that the polynomial $B_{n+m}\left(  z\right)  =B_{n+m}\left(  z;\mathbf{0}\right)  $ admits a
recurrence relation related to the family of $\left\{  z^{i}B_{j}\left(z\right)  \right\}  $ as follows%
\begin{align}
B_{n+m}\left(  z\right)  =\sum_{k=0}^{n}\sum_{j=0}^{m}%
{m\brace j}\dbinom{n}{k}j^{n-k}z^{j}B_{k}\left(  z\right). \label{bb}
\end{align}
Recently, Xu \cite{xu} generalized this result on giving recurrence relation on a large family on Stirling numbers and Mihoubi et al. \cite{mih2} extend the above relation to $r$-Bell polynomials as
\begin{align}
B_{n+m,r}\left( x\right) & =\sum_{k=0}^{n}\sum_{j=0}^{m}{m+r\brace
j+r}_{\!r}\dbinom{n}{k}j^{n-k}x^{j}B_{k,r}\left( x\right) \label{bc}.
\end{align}
Other recurrence relations are given by Mez\H{o} \cite{mez3}.
The following theorem generalizes identities (\ref{bb}), (\ref{bc}), the Carlitz's identities
\cite{car1,car2} given by%
\begin{align*}
& B_{n+m}\left(  1;r\right)  =\sum_{k=0}^{m}{ m+r \brace k+r }
_{\!\!r}B_{n}\left(  1;k+r\right),   \\ & B_{n}\left(  1;r+s\right)
=\sum_{k=0}^{s}%
\genfrac{[}{]}{0pt}{}{s+r}{k+r}%
_{\!r}\left(  -1\right)  ^{s-k}B_{n+k}\left(  1;r\right)
\end{align*}
and shows that $B_{n+m}\left(  z;\mathbf{r}_{p}\right)  $ admits $r$-Stirling
recurrence\ coefficients in the families of basis
\begin{align*}
&\left\{  z^{j}B_{m}\left(  z;\mathbf{r}_{p}+j\mathbf{e}_{p}\right)  :0\leq
j\leq n\right\},   \\ & \left\{  z^{j}B_{m+i}\left(  z;r+j\right)  :0\leq
i\leq\left\vert \mathbf{r}_{p-1}\right\vert ,\ 0\leq j\leq n\right\}  ,
\end{align*}
where $B_{n}\left(  1;r\right)  $ is the number of ways to partition a set of
$n$ elements into non-empty subsets such that the $r$ first elements are in
different subsets.

\begin{theorem}
\label{T6}We have%
\begin{align*}
B_{n+m}\left(  z;\mathbf{r}_{p}\right)   &  =\underset{j=0}{\overset{n}{\sum}}%
{ n+r_p \brace j+r_p }_{\!\!r_p}z^{j}B_{m}\left(  z;\mathbf{r}_{p}+j\mathbf{e}_{p}\right)  ,\\
B_{n+m}\left(  z;\mathbf{r}_{p}\right)   &  =\underset{i=0}{\overset
{\left\vert \mathbf{r}_{p-1}\right\vert }{\sum}}\underset{j=0}{\overset
{n}{\sum}}{ n+r_p \brace j+r_p }
_{\!\!r_p}a_{i}\left(  \mathbf{r}_{p-1}\right)  z^{j}B_{m+i}\left(
z;r_{p}+j\right), \\
z^{n}B_{m}\left(  z;\mathbf{r}_{p}+n\mathbf{e}_{p}\right) & =\sum_{j=0}^{n}%
\genfrac{[}{]}{0pt}{}{n+r_{p}}{j+r_{p}}%
_{\!r_{p}}\left(  -1\right)  ^{n-j}B_{m+j}\left(  z;\mathbf{r}_{p}\right)  .
\end{align*}

\end{theorem}

\begin{proof}
Let $T_{m}\left(  z;\mathbf{r}_{p}\right)  :=\underset{n\geq0}{\sum}\left(
\underset{j=0}{\overset{n}{\sum}}%
{ n+r_p \brace j+r_p }_{\!\!r_p}z^{j}B_{m}\left(  z;\mathbf{r}_{p}+j\mathbf{e}_{p}\right)  \right)
\frac{t^{n}}{n!}.$ \\ \\The second identity given in \cite[Corollary 12]{mih1} by
\begin{align*}
\exp\left(  z\right) B_{m}\left(  z;\mathbf{r}_{p}+\mathbf{e}_{p}\right) &  =  \frac{d}{dz}\left(  \exp\left(  z\right)  B_{m}\left(  z;\mathbf{r}%
_{p}\right)  \right)
\end{align*}
can be used to verify that we have%
\begin{equation}
\exp\left(  z\right)B_{m}\left(  z;\mathbf{r}_{p}+j\mathbf{e}_{p}\right)  =
\frac{d^{j}}{dz^{j}}\left(  \exp\left(  z\right)  B_{m}\left(  z;\mathbf{r}%
_{p}\right)  \right)  .\label{8}%
\end{equation}
Identity (\ref{8}) and the following generating function (see \cite{bro})%
\[
\underset{n\geq j}{\sum}%
{ n+r_p \brace j+r_p }_{\!\!r_p}\frac{t^{n}}{n!}=\frac{1}{j!}\left(  \exp\left(  t\right)  -1\right)
^{j}\exp\left(  r_{p}t\right)
\]
prove that%
\begin{align*}
T_{m}\left(  z;\mathbf{r}_{p}\right)   &  =\underset{j\geq0}{\sum}B_{m}\left(
z;\mathbf{r}_{p}+j\mathbf{e}_{p}\right)  z^{j}\frac{1}{j!}\left(  \exp\left(
t\right)  -1\right)  ^{j}\exp\left(  r_{p}t\right) \\
&  =\exp\left(  r_{p}t-z\right)  \underset{j\geq0}{\sum}\frac{d^{j}}{dz^{j}%
}\left(  \exp\left(  z\right)  B_{m}\left(  z;\mathbf{r}_{p}\right)  \right)
\frac{\left(  z\left(  \exp\left(  t\right)  -1\right)  \right)  ^{j}}{j!}.
\end{align*}
Now, by the Taylor-Maclaurin's expansion we have%
\[
\underset{j\geq0}{\sum}\frac{d^{j}}{dz^{j}}\left(  \exp\left(  z\right)
B_{m}\left(  z;\mathbf{r}_{p}\right)  \right)  \frac{\left(  u-z\right)  ^{j}%
}{j!}=\exp\left(  u\right)  B_{m}\left(  u;\mathbf{r}_{p}\right)  ,
\]
and we get
$T_{m}\left(  z;\mathbf{r}_{p}\right)  =\exp\left(  r_{p}t-z\right)
\exp\left(  z\exp\left(  t\right)  \right)  B_{m}\left(  z\exp\left(
t\right)  ;\mathbf{r}_{p}\right).$
\\ By Theorem \ref{T4}, we obtain
$T_{m}\left(  z;\mathbf{r}_{p}\right)  =
\underset{n\geq0}{\sum}B_{n+m}\left(  z;\mathbf{r}_{p}\right)  \frac{t^{n}}{n!}.$
\\ By identification, the first identity of Theorem follows.\\ The
second identity follows on utilizing Theorem \ref{T3} to replace $B_{m}\left(
z;\mathbf{r}_{p}+j\mathbf{e}_{p}\right)  $ by \ %
\[
\sum_{i=0}^{\left\vert \mathbf{r}_{p-1}\right\vert }a_{i}\left(  \mathbf{r}_{p-1}\right)
B_{m+i}\left(  z;j+r_{p}\right)  .
\]
For the third identity, let
$A:= \sum_{j=0}^{n}\genfrac{[}{]}{0pt}{}{n+r_{p}}{j+r_{p}}
_{\!r_{p}}\left(  -1\right)  ^{n-j}B_{m+j}\left(  z;\mathbf{r}_{p}\right).$
\\ We use the identity (\ref{9}) and $\left(
k+r_{p}\right)  ^{\underline{n}}=\underset{j=0}{\overset{n}{\sum}}
\genfrac{[}{]}{0pt}{}{n+r_{p}}{j+r_{p}}_{\!r_{p}}k^{j},$ see \cite{bro}, to obtain%
\begin{align*}
A &=\exp\left(  -z\right)  \underset{k\geq0}{\sum}P_{m}\left(  k;\mathbf{r}%
_{p}\right)  \frac{z^{k}}{k!}\sum_{j=0}^{n}%
\genfrac{[}{]}{0pt}{}{n+r_{p}}{j+r_{p}}%
_{\!r_{p}}\left(  -1\right)  ^{n-j}\left(  k+r_{p}\right)  ^{j}\\
&  =\left(  -1\right)  ^{n}\exp\left(  -z\right)  \underset{k\geq0}{\sum}%
P_{m}\left(  k;\mathbf{r}_{p}\right)  \frac{z^{k}}{k!}\sum_{j=0}^{n}%
\genfrac{[}{]}{0pt}{}{n+r_{p}}{j+r_{p}}%
_{\!r_{p}}\left(  -k-r_{p}\right)  ^{j}\\
&  =\left(  -1\right)  ^{n}\exp\left(  -z\right)  \underset{k\geq0}{\sum}%
P_{m}\left(  k;\mathbf{r}_{p}\right)  \frac{z^{k}}{k!}\left(  -k-r_{p}%
+r_{p}\right)  ^{\overline{n}}\\
&  =\exp\left(  -z\right)  \underset{k\geq n}{\sum}P_{m}\left(  k;\mathbf{r}%
_{p}\right)  k^{\underline{n}}\frac{z^{k}}{k!}\\
&  =z^{n}\exp\left(  -z\right)  \underset{k\geq0}{\sum}P_{m}\left(
k+n;\mathbf{r}_{p}\right)  \frac{z^{k}}{k!}\\
&  =z^{n}B_{m}\left(  z;\mathbf{r}_{p}+n\mathbf{e}_{p}\right)  .
\end{align*}
\end{proof}

\noindent As consequences of the last theorem, some identities for the $\left(  r_{1},\ldots,r_{p}\right)  $-Stirling
numbers of second kind are given by the following corollary.
\begin{corollary}
We have%
\begin{align*}
\underset{i=0}{\overset{k}{\sum}}%
{m+\left\vert \mathbf{r}_{p}\right\vert \brace i+r_{p}}_{\!\!\mathbf{r}_{p}}%
{n+r_{p} \brace k-i+r_{p}}_{\!\!r_{p}} &  =%
{n+m+\left\vert \mathbf{r}_{p}\right\vert \brace k+r_{p}}_{\!\!\mathbf{r}_{p}},\\
\sum_{j=0}^{n}%
{m+j+\left\vert \mathbf{r}_{p}\right\vert \brace k+n+r_{p}}_{\!\!\mathbf{r}_{p}}%
\genfrac{[}{]}{0pt}{}{n+r_{p}}{j+r_{p}}%
_{\!r_{p}}\left(  -1\right)  ^{n-j} &  =%
{m+\left\vert \mathbf{r}_{p}\right\vert +n \brace k+r_{p}+n}_{\!\!\mathbf{r}_{p}+n\mathbf{e}_{p}}, \\
\sum_{j=0}^{n}%
{m+j+\left\vert \mathbf{r}_{p}\right\vert \brace k+r_{p}}_{\!\!\mathbf{r}_{p}}%
\genfrac{[}{]}{0pt}{}{n+r_{p}}{j+r_{p}}%
_{\!r_{p}}\left(  -1\right)  ^{n-j}&=0,\ \ k<n.
\end{align*}

\end{corollary}

\begin{proof}
From the first identity of Theorem \ref{T6} we have
\begin{align*}
B_{n+m}\left(  z;\mathbf{r}_{p}\right)   =\underset{j=0}{\overset{n}{\sum}}%
{n+r_{p}\brace j+r_{p}}_{\!\!r_{p}}z^{j}B_{m}\left(  z;\mathbf{r}_{p}+j\mathbf{e}_{p}\right)
\end{align*}
which can be written as
\begin{align*}
\underset{j=0}{\overset{n}{\sum}}%
{n+r_{p}\brace j+r_{p}}_{\!\!r_{p}}z^{j}\underset{i=0}{\overset{m+\left\vert \mathbf{r}_{p-1}\right\vert }{\sum}}%
{m+\left\vert \mathbf{r}_{p}\right\vert \brace i+r_{p}}_{\!\!\mathbf{r}_{p}}z^{i}
\\=\underset{k=0}{\overset{n+m+\left\vert \mathbf{r}_{p-1}\right\vert }{\sum}}z^{k}\underset{i=0}{\overset{k}{\sum}}%
{m+\left\vert \mathbf{r}_{p}\right\vert \brace i+r_{p}}_{\!\!\mathbf{r}_{p}}{n+r_{p}\brace k-i+r_{p}}_{\!\!r_{p}}.
\end{align*}
Then, use the definition of $B_{n+m}\left(  z;\mathbf{r}_{p}\right)  ,$ the
desired identity follows by identification. Using the definition of $B_{n}\left(  z;\mathbf{r}_{p}\right)
$ and the third identity of Theorem \ref{T6}, the second and the third identities of
the corollary follow by the fact that
\begin{align*}
\underset{k=0}{\overset{m+\left\vert \mathbf{r}_{p-1}\right\vert }{\sum}}%
B_{m}\left(  z;\mathbf{r}_{p}+n\mathbf{e}_{p}\right)
={m+\left\vert \mathbf{r}_{p}\right\vert +n\brace k+r_{p}+n}_{\!\!\mathbf{r}_{p}+n\mathbf{e}_{p}}z^{k}
\end{align*}
and from the expansion
\begin{align*}
B_{m}\left(  z;\mathbf{r}_{p}+n\mathbf{e}_{p}\right) & =z^{-n}\sum_{j=0}^{n}%
\genfrac{[}{]}{0pt}{}{n+r_{p}}{j+r_{p}}_{\!r_{p}}\left(  -1\right)  ^{n-j}B_{m+j}\left(  z;\mathbf{r}_{p}\right)
\\ &  =z^{-n}\sum_{j=0}^{n} \genfrac{[}{]}{0pt}{}{n+r_{p}}{j+r_{p}}
_{\!r_{p}}\left(  -1\right)  ^{n-j}\underset{k=0}{\overset{m+j+\left\vert
\mathbf{r}_{p-1}\right\vert }{\sum}}%
{m+j+\left\vert \mathbf{r}_{p}\right\vert \brace k+r_{p}}_{\!\!\mathbf{r}_{p}}z^{k}
\\ &  =\underset{k=0}{\overset{m+n+\left\vert \mathbf{r}_{p-1}\right\vert }{\sum
}}z^{k-n}\sum_{j=0}^{n}{m+j+\left\vert \mathbf{r}_{p}\right\vert \brace k+r_{p}}_{\!\!\mathbf{r}_{p}}
\genfrac{[}{]}{0pt}{}{n+r_{p}}{j+r_{p}}_{\!r_{p}}\left(  -1\right)  ^{n-j}
\\ &  =\underset{k=-n}{\overset{m+\left\vert \mathbf{r}_{p-1}\right\vert }{\sum}%
}z^{k}\sum_{j=0}^{n}%
{m+j+\left\vert \mathbf{r}_{p}\right\vert \brace k+n+r_{p}}_{\!\!\mathbf{r}_{p}}
\genfrac{[}{]}{0pt}{}{n+r_{p}}{j+r_{p}}_{\!r_{p}}\left(  -1\right)  ^{n-j}.%
\end{align*}
\end{proof}

\section{Ordinary generating functions }
The \textit{o.g.f.} of the $r$-Stirling numbers of the second kind \cite{bro} is given by%
\begin{align}
\underset{n\geq k}{\sum}%
{n+r\brace k+r}_{\!\!r}t^{n}=t^{k}\underset{j=0}{\overset{k}{%
{\displaystyle\prod}
}}\left(  1-\left(  r+j\right)  t\right)  ^{-1}. \label{zz}
\end{align}
An analogue result for the $\mathbf{r}_{p}$-Stirling numbers is given by the following theorem.

\begin{theorem}
Let%
\[
\widetilde{B}_{n}\left(  z;\mathbf{r}_{p}\right)  :=\underset{k=0}{\overset
{n}{\sum}}%
{n+\left\vert \mathbf{r}_{p}\right\vert \brace k+\left\vert
\mathbf{r}_{p}\right\vert }_{\!\!\mathbf{r}_{p}}z^{k}.
\]
Then, we have%
\begin{align*}
\underset{n\geq k}{\sum}%
{n+\left\vert \mathbf{r}_{p}\right\vert \brace k+\left\vert
\mathbf{r}_{p}\right\vert }
_{\!\!\mathbf{r}_{p}}t^{n}  & =t^{k+\left\vert \mathbf{r}_{p-1}\right\vert
}\left(  \frac{1}{t}\right)  ^{\underline{r_{1}}}\cdots\left(  \frac{1}%
{t}\right)  ^{\underline{r_{p-1}}}\underset{j=0}{\overset{k+\left\vert
\mathbf{r}_{p-1}\right\vert }{%
{\displaystyle\prod}
}}\left(  1-\left(  r_{p}+j\right)  t\right)  ^{-1},\\
\underset{n\geq0}{\sum}\widetilde{B}_{n}\left(  z;\mathbf{r}_{p}\right)
t^{n}  & =\left(  \frac{1}{t}\right)  ^{\underline{r_{1}}}\cdots\left(  \frac{1}{t}\right)
^{\underline{r_{p-1}}}\underset{k\geq\left\vert \mathbf{r}_{p-1}\right\vert
}{\sum}\frac{z^{k-\left\vert \mathbf{r}_{p-1}\right\vert }t ^{k}}{\underset{j=0}{\overset{k}{\prod }}\left( 1-\left( r_{p}+j\right) t\right)   }.
\end{align*}

\end{theorem}

\begin{proof}
Use Corollary \ref{C1} to obtain%
\begin{align*}
\underset{n\geq k}{\sum}%
{n+\left\vert \mathbf{r}_{p}\right\vert \brace k+\left\vert
\mathbf{r}_{p}\right\vert }_{\!\!\mathbf{r}_{p}}t^{n}  & =\underset{n\geq k}{\sum}%
{n+\left\vert \mathbf{r}_{p}\right\vert \brace k+\left\vert
\mathbf{r}_{p-1}\right\vert +r_{p}}_{\!\!\mathbf{r}_{p}}t^{n}\\
& =\underset{j=0}{\overset{\left\vert \mathbf{r}_{p-1}\right\vert }{\sum}%
}a_{j}\left(  \mathbf{r}_{p-1}\right)  t^{-j}\underset{n\geq k}{\sum}%
{n+j+r_{p}\brace k+\left\vert \mathbf{r}_{p-1}\right\vert
+r_{p}}_{\!\!r_{p}}t^{n+j}\\
& =\underset{j=0}{\overset{\left\vert \mathbf{r}_{p-1}\right\vert }{\sum}%
}a_{j}\left(  \mathbf{r}_{p-1}\right)  t^{-j}\underset{n\geq k+j}{\sum}%
{n+r_{p}\brace k+\left\vert \mathbf{r}_{p-1}\right\vert
+r_{p}}_{\!\!r_{p}}t^{n},
\end{align*}
and because ${n+r_{p}\brace k+\left\vert \mathbf{r}_{p-1}\right\vert
+r_{p}}_{\!\!r_{p}}=0$ for $n=k,\ldots,+\left\vert \mathbf{r}_{p-1}\right\vert-1,$ we get
\begin{align*}
\underset{n\geq k}{\sum}%
{n+\left\vert \mathbf{r}_{p}\right\vert \brace k+\left\vert
\mathbf{r}_{p}\right\vert }_{\!\!\mathbf{r}_{p}}t^{n}
& =\left(  \underset{n\geq k+\left\vert \mathbf{r}_{p-1}\right\vert }{\sum}%
{n+r_{p}\brace k+\left\vert \mathbf{r}_{p-1}\right\vert
+r_{p}}_{\!\!r_{p}}t^{n}\right)  \left(  \underset{j=0}{\overset{\left\vert
\mathbf{r}_{p-1}\right\vert }{\sum}}a_{j}\left(  \mathbf{r}_{p-1}\right)
t^{-j}\right)  .
\end{align*}
The first generating function of the theorem follows by using (\ref{10}) and  (\ref{zz}).
For the second one, use the definition of $\widetilde{B}_{n}\left(z;\mathbf{r}_{p}\right)  $ and the last expansion.
\end{proof}

\section{Acknowledgments}
The authors would like to acknowledge the support from the PNR project 8/U160/3172 and the RECITS's laboratory.

\end{document}